\documentclass[10pt]{article}
\usepackage[utf8]{inputenc}
\usepackage{subfiles}
\usepackage{amsmath}
\usepackage{verbatim}
\usepackage{multirow}
\usepackage{amsmath}
\usepackage{amssymb}
\usepackage{amsthm}
\usepackage{mathrsfs}
\usepackage{hyperref}
\usepackage{amsmath}
\usepackage[shortlabels]{enumitem}
\usepackage{amssymb}
\usepackage{bbm}
\usepackage{tikz-cd}
\usepackage{array}
\usepackage{polynom}
\usepackage{geometry}
\usepackage{mathtools}
\usepackage{setspace}
\usepackage{accents}
\usepackage{xcolor}
\usepackage{soul}

\theoremstyle{plain}
\newtheorem{theorem}{Theorem}[section]
\newtheorem{proposition}{Proposition}[section]
\newtheorem{definition}{Definition}[section]

\newtheorem{remark}{Remark}

\newtheorem{lemma}{Lemma}[section]
\newtheorem{corollary}{Corollary}[section]

\newcommand{\Z}{\mathbb{Z}}

\newcommand{\C}{\mathbb{C}}

\newcommand{\gl}{\frak{gl}}

\usepackage{etoc}
\usepackage{float}
\usepackage{blindtext}
\usepackage{authblk}

\graphicspath{{images/}{../images/}}

\title{Rational Schur Superalgebras}
\author{Andrew Riesen  }
\renewcommand\footnotemark{}

\begin{document}
\maketitle
\begin{abstract}
We develop and study the generalization of rational Schur algebras to the super setting. Similar to the classical case, this provides a new method for studying rational supermodules of the general linear supergroup $GL(m|n)$. Furthermore, we establish a Schur-Weyl duality result for rational Schur superalgebras and conclude that under certain conditions these objects will be semisimple.
\end{abstract}
\smallskip

\section{Introduction}
In the 1980s, Green demonstrated that understanding the polynomial representations of $GL(n)$ over $\C$ is intimately linked to understanding the representation theory of Schur algebras \cite{green}. More precisely, he showed that any finite-dimensional polynomial representation of $GL(n)$ over $\C$ can be split into homogeneous representations and the category of such degree-$r$ homogeneous representations is equivalent to the category of modules of the Schur algebra $S(n,r)$. In \cite{rational}, Dotty and Dipper presented a generalization of Schur algebras called rational Schur algebras. These were defined over a field $k$ of arbitrary characteristic and shown to have a similar connection with rational representations of $GL(n)$. A key component to understanding the structure of (rational) Schur algebras is an equivalence of three definitions. For example, the Schur algebra $S(n,r)$ is defined as the dual of degree-$r$ homogeneous polynomials on $GL(n)$. It is isomorphic to the centralizer of the action of the symmetric group $S_r$ on $(k^n)^{\otimes r}$, and to the image $\rho_r(U(\gl(n)))$, where $\rho_r$ denotes the canonical representation of $\gl(n)$ on $(k^n)^{\otimes r}$. Analogous definitions are equivalent for the rational Schur algebra. In this paper, we generalize the connection rational Schur algebras have with the general linear group and demonstrate a similar equivalence of definitions in the context of supermathematics.

 Supermathematics was first studied in the 1980s and arose from the mathematical framework required for supersymmetry in physics. It has been a theme of modern mathematics to develop ideas from representation theory to the super case and study how the properties change or stay the same. We continue this theme by studying rational Schur algebras in the super case, which we call rational Schur superalgebras. Similar notions were also considered in \cite{marko2020}.

Let $m,n\geq 0$. The super analogue of the general linear group is called the general linear supergroup, denoted by $GL(m|n)$. Contrary to the classical case, $GL(m|n)$ can only be described as a particular representable functor. In addition, properties that hold in the classical case fail in the super case, such as complete reducibility. Due to this, the representation theory of the general linear supergroup is more difficult to understand. Though one can still consider the super analogue of polynomial and rational representations of $GL(m|n)$, unlike in the classical case, not all rational representations can be obtained from polynomial representations alone. This is because the super analogue of the determinant, the Berezinian, is not a polynomial function of $GL(m|n)$. Thus, the classical argument of tensoring a finite-dimensional rational representation by a large power of the determinant representation will no longer produce a polynomial representation in the super setting. For this reason, we believe that understanding the super analogue of the rational Schur algebra will provide an important tool for studying rational representations of $GL(m|n)$.
  
This paper is organized as follows. First, we review notions and definitions from supermathematics. Then we introduce a superbialgebra $\tilde{\mathcal{A}}(m|n)$, which will play the role of the polynomial ring $k[x_{i,j}:1\leq i,j\leq n]$ in the classical setting, and study its structure. From this superbialgebra we obtain rational Schur superalgebras, introduce their category of supermodules, and clarify the connection they have to representations of the general linear supergroup. The rest of this paper is devoted to proving a semisimplicity result for rational Schur superalgebras and an equivalence of three definitions. We approach this problem by establishing a Hopf superalgebra isomorphism between $\tilde{\mathcal{A}}(m|n)$ and the Hopf superalgebra of regular functions on $GL(m|n)$, which was introduced by Scheunnert and Zhang in \cite{SCHEUNERT200244}. This isomorphism then allows us to prove an equivalence of definitions and the semisimplicity of rational Schur superalgebras under certain conditions. 

In future work we hope to generalize the results of this paper to all rational Schur superalgebras of a fixed $(m,n)$, similar to what has been achieved for Schur superalgebras. By combining Theorem $3.1$ with the remarks in Section $3.3$, this should provide a new way to understand the representation theory of $GL(m|n)$.
\section{Preliminaries}

First, recall that a superspace $V$ over a field $k$ is a $k$-vector space with a $\Z/2$-grading. That is, $V=V_0\oplus V_1$, where elements of $V_0$ are referred to as even elements and elements of $V_1$ as odd. If $x\in V_0$ or $x\in V_1$ it is called homogeneous and its parity is defined by
\begin{equation*}|x|=\begin{cases}0&\text{ if }x\in V_0\\ 1&\text{ if }x\in V_1\end{cases}\end{equation*}

If $V$ is a finite-dimensional superspace such that $\dim(V_0)=m$ and $\dim(V_1)=n$, then its super-dimension is referred to as $(m|n)$. For example, the standard $(m|n)$-dimensional superspace, $k^{m|n}$, is defined as $k^{m}\oplus k^n$.

A morphism of superspaces $V$ and $W$ is a $k$-linear map $\phi:V\rightarrow W$ such that $\phi(V_i)\subseteq W_i$ for $i=0,1$. Such morphisms will be referred to as even linear maps. 

Given superspaces $V$ and $W$, we may endow their tensor product $V\otimes_k W$ with a superspace structure by defining the even and odd components as 
\[(V\otimes_k W)_0=(V_0\otimes_kW_0)\oplus (V_1\otimes_k W_1)\] 
\[(V\otimes_k W)_1= (V_0\otimes_k W_1)\oplus (V_1\otimes_k W_0)\]
An associative $k$-superalgebra is then a superspace $A$ which is an associative $k$-algebra such that its structure maps:
\begin{equation*}
 \Delta:A\otimes_k A\rightarrow A, \quad\ \eta:k\rightarrow A
\end{equation*} 
are even linear maps. Since we will be working exclusively with associative $k$-superalgebras we will simply refer to them as superalgebras. A homomorphism of superalgebras $\phi:A\rightarrow B$ is then a homomorphism of associative algebras such that $\phi$ is an even linear map. We say that a superalgebra $A$ is supercommutative if for all homogeneous $x, y\in A$ the following holds:
\[xy=(-1)^{|x|\cdot|y|}yx\]

 Also, if $A,B$ are superalgebras, then so is $A\otimes_k B$ with multiplication defined on homogeneous elements $x,w\in A$ and $y,z\in B$ by
\[(x\otimes y)\cdot (w\otimes z)=(-1)^{|y||w|}((x w)\otimes (y z))\]
and then extended linearly to all of $A\otimes_k B$.

Supercoalgebras, supermodules, supercomodules and their morphisms are defined in the exact same way as their classical counterparts, except the structure maps must respect the $\Z/2$-gradings as demonstrated above by the superalgebra definition. See  \cite{Supermath} for details. 

\subsection{The General Linear Supergroup}
In this section we recall for the sake of completeness  some definitions and facts pertaining to the general linear supergroup. All superspaces and superalgebras are assumed to be over a fixed field $k$. See the introductions of \cite{zubkov2006some, kujawa2004representation} for a more complete treatment. 

Contrary to the classical case, the general linear supergroup of a superspace is most naturally defined as a functor, in the following way. Let $M$ be a finite-dimensional superspace and fix a homogeneous basis $\{m_1,\hdots,m_t\}$. Then $M\otimes_k A$ is a right $A$-supermodule and has a $A$-linear basis given by $\mathcal{C}=\{m_i\otimes 1:1\leq i\leq t\}$. The general linear supergroup $GL(M)$ is a functor from the category of finite dimensional supercommutative superalgebras to the category of groups. It associates to every supercommutative superalgebra $A$ all of the invertible even $A$-linear maps on $M\otimes_k A$, which we denote by $GL(M)(A)$. Furthermore, if $f:A\rightarrow B$ is a morphism of supercommutative superalgebras, then $GL(M)$ sends it to $GL(M)(f):GL(M)(A)\rightarrow GL(M)(B)$, which is defined by writing an element of $T\in GL(M)(A)$ as a $t\times t$ matrix (with respect to the basis $\mathcal{C}$) with entries in $A$ and applying $f$ to each entry. In symbols, the general linear supergroup $GL(M)$ is the functor which sends objects 
\[A\mapsto (\operatorname{End}_A(M\otimes_k A))^{\times}\]
and morphisms 
\[f\in \operatorname{Hom}_{\bold{superalg}}(A,B)\ \text{ to }\ \ GL(M)(f)((a_{i,j})_{1\leq i,j\leq t})=(f(a_{i,j}))_{1\leq i,j\leq t}\]
This definition is independent of the chosen basis of $M$. 

In general $GL(M)$ will be a representable functor, but of particular interest to us is the case where $M=k^{m|n}$. The representing object of $GL(k^{m|n})$ is referred to as its coordinate superalgebra and is defined in the following way. 
\begin{definition}{\cite{zubkov2006some}}
Let $\mathcal{A}(m|n)=k[x_{ij}:1\leq i,j\leq m+n]$ be the free supercommutative superalgebra on $(m+n)\times (m+n)$ variables, with the parity on the variables given by \[|x_{ij}|:=\overline{i}+\overline{j}\mod 2,\text{ where  }\overline{i}=\begin{cases} \overline{0} 
&\text{ if }1\leq i\leq m\\ \overline{1}&\text{ if }m+1\leq i\leq m+n\end{cases}\]

 Let $d_1=\det(x_{ij})_{1\leq i,j\leq m}$ and $d_2=\det(x_{ij})_{m+1\leq i,j \leq m+n}$, which are even elements of $\mathcal{A}(m|n)$. The \emph{coordinate superalgebra of $GL(k^{m|n})$} is then defined as the localization of $\mathcal{A}(m|n)$ at $d_1 d_2$.
\label{coord_ring}
\end{definition}
From now on we will denote  $GL(k^{m|n})$  by $GL(m|n)$ and its coordinate superalgebra by $k[GL(m|n)]$.
Since the general linear supergroup is defined as a functor, it is natural for its supermodules to be morphisms of functors.
\begin{definition}
Let $M$ be a finite-dimensional superspace. Then it is a \emph{left $GL(m|n)$-supermodule} if there is a natural transformation $p:GL(m|n)\rightarrow GL(M)$.
\label{supergroup_rep4}
\end{definition}

The coordinate superalgebra of $GL(m|n)$ is also a Hopf superalgebra, with comultiplication and counit structure maps given respectively by:
\[\Delta(x_{i,j})=\sum_{k=1}^{n+m}x_{i,k}\otimes x_{k,j}, \ \ \ \ \epsilon(x_{i,j})=\delta_{i,j}\]
As described in \cite{brundanmodular} there is an equivalence between the category of left $GL(m|n)$-supermodules and right $k[GL(m|n)]$-supercomodules. We use this identification implicitly from now on when talking about $GL(m|n)$-supermodules. 

The following definitions and facts will be needed to define rational Schur superalgebras. 
\begin{definition}
Suppose that $M$ is a finite-dimensional right supercomodule of $k[GL(m|n)]$, with structure map $\tau:M\rightarrow M\otimes_k k[GL(m|n)]$. Let $\{v_1,..,v_r\}$ be a homogeneous basis of $M$. Then $\tau(v_j)=\sum_{i=1}^rv_i\otimes a_{ij}$, with $a_{ij}\in k[GL(m|n)]$. The matrix formed from these $k[GL(m|n)]$ coefficients, $A:=(a_{ij})_{1\leq i,j\leq r}$, is called the \emph{defining matrix} of the representation. 
\end{definition}
\begin{definition}
Let $M$ be a right supercomodule of $k[GL(m|n)]$. Then the \emph{coefficient space of $M$}, denoted $cf(M)$, is the superspace spanned by the entries of its defining matrix. 
\end{definition}
These constructions are independent of the choice of basis. Furthermore, if $M,N$ are representations of $GL(m|n)$, then $M\otimes_k N$ is as well and the following equality holds in $k[GL(m|n)]$:
\[cf(M\otimes_k N)=cf(M)\cdot cf(N) \] 

\subsection{Coefficient Spaces of Superalgebras}
The following results will be needed in Section $4$.

 Let $M$ be a finite-dimensional superspace and $A$ a superalgebra. Using the identification $M\otimes_k M^*\cong \operatorname{End}_k(M)$ we may endow $\operatorname{End}_k(M)$ with the structure of a superspace. It can be easily checked that composition turns $\operatorname{End}_k(M)$ into a superalgebra. So, by a representation $M$ of $A$ we mean a superalgebra homomorphism $\rho:A\rightarrow \operatorname{End}_k(M)$. 

\begin{definition}
Let $M$ be a finite-dimensional superspace with homogeneous basis $B_M=\{v_1,...,v_{\ell}\}$. If $\rho:A\rightarrow \operatorname{End}_k(M)$ is a representation of $A$, then with respect to the basis $B_M$, $\rho(a)$ is an $\ell \times \ell$ matrix with entries in $k$. Let $1\leq i,j\leq \ell$. Then the $(i,j)$th entry of $\rho(a)$ depends linearly on $a$ and so defines an element of $A^*$, which we denote as $a_{i,j}$. The \emph{defining matrix of $\rho$} is the matrix $(a_{i,j})_{1\leq i,j\leq \ell}\in \operatorname{Mat}_{\ell \times \ell}(A^*)$.
\end{definition}

\begin{definition}
\label{coeff_alg}
Let $A,M,\rho$ be as above. Then the \emph{coefficient space of $M$} is
\[cf_A(M)=\operatorname{span}_{k}\{a_{i,j}:1\leq i,j\leq l\}\]
\end{definition}
It is well known that $cf_A(M)$ is independent of the choice of basis of $M$ and so $cf_A(M)$ is well-defined. 
We will also need the following fact observed in \cite[Pg. 53]{SCHEUNERT200244}.
\begin{lemma}
\label{coeff_alg_iso}
Let $\rho:A\rightarrow \operatorname{End}_k(M)$ be a finite-dimensional representation of $A$. Then $(cf_A(M))^*$ is a superalgebra that is isomorphic to $\rho(A)$.
\end{lemma}
\section{Category of Supermodules}
Throughout this section, $k$ will denote an algebraically closed field of characteristic not equal to $2$, and all $GL(m|n)$-supermodules are left $GL(m|n)$-supermodules.
\subsection{Polynomial Representations of $GL(m|n)$}
To motivate the results for rational representations of $GL(m|n)$, we recall some basic facts about $\mathcal{A}(m|n)$ and define what polynomial representations of $GL(m|n)$ are.

Let $\mathcal{A}(m|n)$ be as in Definition \ref{coord_ring}. Like the coordinate superalgebra $k[GL(m|n)]$, it is a superbialgebra \cite{muir1991polynomial}.

\begin{definition}

For a positive integer $r$, define the \emph{degree-$r$ homogeneous polynomials} of $\mathcal{A}(m|n)$ to be:

\[\mathcal{A}(m|n;r):=\operatorname{span}_k\{x_{i_i,j_1}...x_{i_r,j_r}:1\leq i_r,j_r\leq m+n\}\]
If $r=0$, then $\mathcal{A}(m|n;0):=k$.
\end{definition}

It follows easily from the definition of comultiplication in $k[GL(m|n)]$ that $\mathcal{A}(m|n;r)$ is a sub-supercoalgebra of $k[GL(m|n)]$. As in the classical case we have a supercoalgebra decomposition:

 \[ \mathcal{A}(m|n)=\bigoplus_{r=0}^{\infty} \mathcal{A}(m|n;r)\]
\begin{definition}

Let $M$ be a $GL(m|n)$-supermodule. Then $M$ is a \emph{polynomial representation} if 
$cf(M)\subseteq \mathcal{A}(m|n)$. 

A \emph{degree-$r$ homogeneous representation} of $GL(m|n)$ is a $GL(m|n)$-supermodule $M$ such that $cf(M)\subseteq \mathcal{A}(m|n;r)$.
\end{definition}

\subsection{Rational Representations of Bidegree $(r,s)$}


Suppose that $M$ is a left $GL(m|n)$-supermodule  and let

\[i:GL(m|n)\rightarrow GL(m|n)\]
be the inversion morphism. This induces a dual map $i^*:k[GL(m|n)]\rightarrow k[GL(m|n)]$ which is the anti-pode of $k[GL(m|n)]$.
Since defining rational Schur superalgebras requires the map $i^*$ and coefficient spaces from a dual representation, we will discuss these objects in more detail. 
Consider the $(m+n)\times (m+n)$ matrix $(x_{k,\ell})_{1\leq k,\ell\leq m+n}$ with entires in $k[GL(m|n)]$. Then $i^*(x_{p,q})$ is the $(p,q)th$ entry of $((x_{k,\ell})_{1\leq k,\ell \leq m+n})^{-1}$, and these entries will be in $k[GL(m|n)]$. 
If $C$ is the defining matrix of a supercomodule $M$ of $k[GL(m|n)]$, then $M^*$ is a supercomodule of $k[GL(m|n)]$ and $(C^{-1})^{\bold{st}}$ is its defining matrix  \cite{zubkov2006some} . Here $\bold{st}$ denotes the supertranspose:

\begin{definition} Suppose that $V$ is a finite-dimensional $A$-supermodule with ordered homogeneous basis $B_V=(v_1,...,v_{\ell}|\overline{v_1},...,\overline{v_{\ell'}})$, and $T\in \operatorname{End}_A(V)$ is even. Write $[T]_{B_V}=\begin{bmatrix}T_1&T_2\\ T_3 & T_4\end{bmatrix}$. Then the \emph{supertranspose} of $[T]_{B_V}$ is \[[T]_{B_V}^{st}=\begin{bmatrix}(T_1)^t & (T_3)^t\\ -(T_2)^t& (T_4)^t\end{bmatrix}\]
\label{super_transp}
\end{definition}

For the sake of brevity, denote $V=k^{m|n}$ and $W=V^*$. Then $V$ is a left $GL(m|n)$-supermodule with structure map given by 
\[\tau(e_i)=\sum_{k=1}^{m+n} e_k\otimes x_{ki}\]
Evidently, $cf(V^{\otimes r})=\mathcal{A}(m|n;r)$.
\begin{definition}
Let $V$ and $W$ be as above. Then the $(r,s)$-homogeneous component of $k[GL(m|n)]$ is
\[\tilde{\mathcal{A}}(m|n;r,s):=cf(V^{\otimes r}\otimes_kW^{\otimes s})=cf(V)^{r}cf(W)^s\] 
The space generated by all such components is
 \[\tilde{\mathcal{A}}(m|n):=\sum_{(r,s)\in \Z_{\geq 0}^2}\tilde{\mathcal{A}}(m|n;r,s)\]
\end{definition}
We would like to describe what elements of $\tilde{\mathcal{A}}(m|n;r,s)$ look like for fixed $(r,s)$. To do this we need the notion of a parity functor and the super analogue of the determinant, the Berezinian.
The parity functor $\Pi$ is a functor on superspaces which takes a superspace $V=V_0\oplus V_1$ to the superspace $\Pi(V)$ with even component $V_1$ and odd component $V_0$, while fixing all even linear maps. In other words, it switches the parity of superspaces.
\begin{definition}
Let $T, V$ be as in Definition \ref{super_transp}, and set $B_{\Pi(V)}=(\overline{v_1},...,\overline{v_{\ell'}}|v_1,...,v_{\ell})$. Then the \emph{Berezinian} is  
\[Ber(T):=\det(T_1-T_2T_4^{-1}T_3)\det(T_4)^{-1}\]
If $T$ is an even $A$-linear map, then $[\Pi(T)]_{B_{\Pi(V)}}=\begin{bmatrix}T_4&T_3\\ T_2 & T_1\end{bmatrix}$, and
\begin{equation*}Ber^*(T):=Ber(\Pi(T))\end{equation*}
\end{definition}
\noindent This definition is independent of the homogeneous basis chosen for $V$.
\begin{proposition}{\cite{khudaverdian2005berezinians}}
\begin{enumerate}
    \item $Ber(ST)=Ber(S)Ber(T)$
    \item $Ber^*(ST)=Ber^*(S)Ber^*(T)$
\end{enumerate}
\end{proposition}

Recall that $\mathcal{A}(m|n)=k[x_{ij}:1\leq i,j\leq m+n]$, and consider the matrix $A=(x_{i,j})_{1\leq i,j\leq m+n}$. Then $Ber, Ber^*$ are elements of $k[GL(m|n)]$ represented by $Ber(A),Ber^*(A)$. Let $D(i,j)$ be the matrix formed by replacing the $i$th row of $A$ with zeros except the $(i,j)$th entry which is $1$. Evidently, $Ber(D(i,j)),Ber^*(D(i,j))\in k[GL(m|n)]$, and we denote these elements by $Ber_{i,j},Ber^*_{i,j}$ respectively. Lastly, denote by $\tilde{x}_{i,j}$ the $(i,j)$th entry of $A^{-1}$. 

We may now give an explicit description of $\tilde{\mathcal{A}}(m|n)$.
\begin{lemma}
\label{description}
The following properties hold:
\begin{enumerate}[ref=\thelemma\text{, Part }\arabic*]
    \item $\tilde{\mathcal{A}}(m|n)$ is generated as a $k$-superalgebra by the elements $x_{i,j} , \tilde{x}_{i,j}$ \label{description1}
    \item By the super Cramer rule \cite{khudaverdian2005berezinians}, $cf(W)$ is the $k$-space spanned by $\tilde{x}_{i,j}$ where \label{description2}
    \[\tilde{x}_{i,j}=\begin{cases} Ber_{i,j}(Ber)^{-1} & \text{ when }i\in \{1,...,m\}\\ Ber^*_{i,j}(Ber^*)^{-1}& \text{ when }i\in\{m+1,...,m+n\} \end{cases} \]

    \item $\tilde{\mathcal{A}}(m|n;r,s)$ is the span of the following set \label{description3}
    \[\{x_{i_1,j_1}^{m_1} \hdots x_{i_t,j_t}^{m_t}\tilde{x}_{i_{t+1},j_{t+1}}^{m_{t+1}} \hdots \tilde{x}_{i_{t+k},j_{t+k}}^{m_{t+k}}: 1\leq i_l,j_l\leq m+n, \ m_1+...+m_t=r, m_{t+1}+...+m_{t+k}=s\}\]
    \item \[\sum_{k=1}^{m+n}x_{i,k}\tilde{x}_{k,j}=\delta_{i,j}\ \ \ \ \ \text{ and } \ \ \ \ \ \sum_{k=1}^{m+n}\tilde{x}_{i,k}x_{k,j}=\delta_{i,j}\] \label{description4}
\end{enumerate}

\end{lemma}
\noindent Note that $Ber^*$ is the multiplicative inverse of $Ber$, but in general $Ber^*_{i,j}$ and $Ber_{i,j}$ are not invertible (see the example at the end of section eight in \cite{khudaverdian2005berezinians}). 

We would like to establish a supercoalgebra decomposition of $\tilde{\mathcal{A}}(m|n)$ in terms of $\tilde{\mathcal{A}}(m|n;r,s)$, which is similar to the one given for $\mathcal{A}(m|n)$. To do this we need to establish that $\tilde{\mathcal{A}}(m|n;r,s)$ is a supercoalgebra and to that end we prove the following:
\begin{lemma}{(Compare with \cite[Equation 2.2.2]{rational})}
\[\Delta(\tilde{x}_{i,j})=\sum_{k=1}^{m+n}(-1)^{(|i|+|k|)(|k|+|j|)}(\tilde{x}_{k,j}\otimes \tilde{x}_{i,k}) \ \ \ \epsilon(\tilde{x}_{i,j})=\delta_{i,j}\]
\label{comulti}
\end{lemma}
\begin{proof}
Suppose that $1\leq i,j, s,t\leq m+n$, and set $C_{i,j}:=\sum_{k=1}^{m+n}(-1)^{(|i|+|k|)(|k|+|j|)}\tilde{x}_{k,j}\otimes \tilde{x}_{i,k}$. Then 
\[\Delta(x_{s,t})C_{i,j}=(\sum_{p=1}^{m+n}x_{s,p}\otimes x_{p,t})(\sum_{k=1}^{m+n}(-1)^{(|i|+|k|)(|k|+|j|)}\tilde{x}_{k,j}\otimes \tilde{x}_{i,k})=\]
\[\sum_{p,k=1}^{m+n}(-1)^{(|i|+|k|)(|k|+|j|)}(-1)^{(|p|+|t|)(|k|+|j|)}(x_{s,p}\tilde{x}_{k,j})\otimes (x_{p,t}\tilde{x}_{i,k})\]
Using this calculation and the relation $\sum_{p=1}^{m+n}x_{i,p}\tilde{x}_{p,j}=\delta_{i,j}$ we find
\begin{align*}
\sum_{w=1}^{m+n}\Delta(x_{s,w})C_{w,j} &=\sum_{w=1}^{m+n}\sum_{p,k=1}^{m+n}(-1)^{(|p|+|w|)(|k|+|j|)+|w|(|k|+|j|)+|k|(|k|+|j|)}(x_{s,p}\tilde{x}_{k,j})\otimes (x_{p,w}\tilde{x}_{w,k})\\
&=\sum_{p,k=1}^{m+n}\sum_{w=1}^{m+n}(-1)^{|p|(|k|+|j|)+|k|(|k|+|j|)}(x_{s,p}\tilde{x}_{k,j})\otimes (x_{p,w}\tilde{x}_{w,k})\\
&=\sum_{p,k=1}^{m+n}(-1)^{(|p|+|k|)(|k|+|j|)}(x_{s,p}\tilde{x}_{k,j})\otimes(\sum_{w=1}^{m+n}x_{p,w}\tilde{x}_{w,k})\\
&=\sum_{p,k=1}^{m+n}(-1)^{(|p|+|k|)(|k|+|j|)}(x_{s,p}\tilde{x}_{k,j})\otimes(\delta_{p,k})\\
&=\sum_{p=1}^{m+n}(-1)^{(|p|+|p|)(|p|+|j|)}(x_{s,p}\tilde{x}_{p,j})\otimes 1= (\sum_{p=1}^{m+n}x_{s,p}\tilde{x}_{p,j})\otimes 1=\delta_{s,j}
\end{align*}

Consider now the $(m+n)\times (m+n)$ matrix $A=(\Delta(x_{i,j}))_{1\leq i,j\leq m+n}$ with entries in $k[GL(m|n)]\otimes_kk[GL(m|n)]$. This has an inverse given by $(\Delta(\tilde{x}_{i,j}))_{1\leq i,j\leq m+n}$ as comultiplication is a superalgebra morphism. If $C$ denotes the matrix $(C_{i,j})_{1\leq i,j\leq m+n}$ then the calculation above shows that $AC=Id$. This implies that $C=(\Delta(\tilde{x}_{i,j}))$ and proves the first part of the lemma. The second part follows in a similar way.
\end{proof}

\begin{proposition}
\label{superco}
$\tilde{\mathcal{A}}(m|n;r,s)$ is a supercoalgebra and $\tilde{\mathcal{A}}(m|n)$ is a superbialgebra 
\end{proposition}
\begin{proof}
This follows immediately from Lemma \ref{comulti} and Lemma \ref{description3}. \end{proof}

 Lemma \ref{description4} implies that 
\[\tilde{\mathcal{A}}(m|n;r,s)\subseteq \tilde{\mathcal{A}}(m|n;r+1,s+1)\]
 Following \cite{rational} we define for $z\in \Z_{\geq 0}$:
\[\tilde{\mathcal{A}}(m|n)_z:=\bigcup_{t\geq 0} \tilde{\mathcal{A}}(m|n;z+t,t)\text{  and }\tilde{\mathcal{A}}(m|n)_{-z}:=\bigcup_{t\geq 0}\tilde{\mathcal{A}}(m|n;t,z+t)\]
Evidently, these are supercoalgebras. In the aforementioned paper, the authors provide a coalgebra decomposition for the classical case $\tilde{\mathcal{A}}(m|0)$  \cite[2.6.1]{rational}. The super analogue of this result holds as well
\begin{equation}\tilde{\mathcal{A}}(m|n)=\bigoplus_{z\in\Z} \tilde{\mathcal{A}}(m|n)_z \label{GL_co_decomp}\end{equation}
The proof of this is similar to the classical one, which we now briefly outline. The only way to realize $w\in \tilde{\mathcal{A}}(m|n;r,s)$ as an element of $\tilde{\mathcal{A}}(m|n;e,f)$ is to repeatedly use the relations $w(\sum_{k=1}^{m+n}x_{i,k}\tilde{x_{k,i}})=w\cdot 1$ or $(\sum_{k=1}^{m+n}\tilde{x}_{i,k}x_{k,i})w=1\cdot w$. This means there must exist some $t\geq 0$, such that $r+t=e,s+t=f$. Thus, if $r\geq s$ this implies that $w\in \tilde{\mathcal{A}}(m|n)_{r-s}$, and similarly if $s>r$ then $w\in \tilde{\mathcal{A}}(m|n)_{-(s-r)} $. So in either case, $\tilde{\mathcal{A}}(m|n;r,s), \tilde{\mathcal{A}}(m|n;e,f)\subseteq \tilde{\mathcal{A}}(m|n)_{r-s}$. This shows the sum in Equation \ref{GL_co_decomp} is direct. 
 
 \begin{definition}
If $M$ is a $GL(m|n)$-supermodule and $cf(M)\subseteq \tilde{\mathcal{A}}(m|n)_z$, then it is a supermodule of \emph{rational degree $z$}. Furthermore, if $cf(M)\subseteq \tilde{\mathcal{A}}(m|n;r,s)$ and $cf(M)\not\subset\tilde{\mathcal{A}}(m|n;r-1,s-1)$, then it is a supermodule of \emph{bidegree $(r,s)$}.
 \end{definition}

\begin{theorem}{(Compare with \cite[Theorem 2.7]{rational})}
If $V$ is a finite-dimensional $GL(m|n)$-supermodule, then the following holds. 
\begin{enumerate}
\item $V=\bigoplus_{z\in \Z}V_{z}$, where $V_z$ is a supermodule of rational degree $z$.
\item Each $V_{z}$ is the union of an ascending chain of $GL(m|n)$-supermodules as follows:
\begin{enumerate}
\item If $z\geq 0$, then $V_{z,0}\subseteq V_{z+1,1}\subseteq ....\subseteq V_{z+t,t}\subseteq ...$, where $V_{z+t,t}$ is a supermodule of bidegree $(z+t,t)$ and $V_z=\bigcup_{t=0}^{\infty}V_{z,t}$.
\item If $z<0$, then $V_{0,-z}\subseteq V_{1,-z+1}\subseteq ...\subseteq V_{t,-z+t}\subseteq ....$, where $V_{t,-z+t}$ is a supermodule of bidegree $(t,-z+t)$ and $V_z=\bigcup_{t=0}^{\infty}V_{t,-z+t}$.
\end{enumerate}
\end{enumerate}
\end{theorem}
\begin{proof}
Lemma 8.3 from \cite{marko2020}, states that $x_{i,j},\tilde{x}_{i,j}$ for $1\leq i,j\leq m+n$ generate $k[GL(m|n)]$. Thus, $\tilde{\mathcal{A}}(m|n)=k[GL(m|n)]$ (See also Corollary \ref{map_gen} for a self-contained proof of this fact). Similar to Dipper and Doty's proof of Theorem 2.7 in \cite{rational}, by combining Equation \ref{GL_co_decomp} with Theorem 1.6c from \cite{Green_result}  we obtain the $k[GL(m|n)]$-comodule decomposition
\[V=\bigoplus_{z\in \Z}V_z\] 
This comodule decomposition is actually a supercomodule decomposition. For if $\tau:V\rightarrow V\otimes_k k[GL(m|n)]$ denotes the supercomodule structure map, then in Green's proof of Theorem 1.6c, $V_z=\tau^{-1}(V\otimes_k  \tilde{\mathcal{A}}(m|n)_z)$. Since the pre-image of a superspace under an even linear map is a superspace, we see that $V_z$ is a supercomodule. Thus completing the proof of the first statement.

 The proof of the second statement is essentially identically to Dipper and Doty's proof in Theorem 2.7 of \cite{rational}, and so is omitted.

\end{proof}

\begin{definition}
Fix non-negative integers $m,n,r,s$. The \emph{rational Schur superalgebra}, which we denote as $\mathcal{S}(m|n;r,s)$, is $(\tilde{\mathcal{A}}(m|n;r,s))^*$.
\end{definition}
\noindent By Proposition \ref{superco} and Lemma \mbox{\ref{description3}} this will indeed be a finite-dimensional superalgebra. 
\subsection{Equivalence of Categories}
It is well know that if $A$ is a finite-dimensional supercoalgebra, then its category of supercomodules is equivalent to the category of supermodules over $A^*$. As in the classical case we define,

\[\mathcal{S}(m|n)_z:=(\tilde{\mathcal{A}}(m|n)_z)^*\]
This equals the following inverse limit when $z\geq 0$ or $z<0$ respectively:
\[\varprojlim_{t\in\Z_{\geq 0}}\mathcal{S}(m|n;z+t,t),\ \ \ \ \ \ \varprojlim_{t\in\Z_{\geq 0}}\mathcal{S}(m|n;t,-z+t)\]
Since every $GL(m|n)$-supermodule can be decomposed into rational supermodules of varying degrees, this suggests that understanding the category of supercomodules of $\tilde{\mathcal{A}}(m|n)_z$ will give some insight into the representation theory of $GL(m|n)$. As we have just mentioned, this is equivalent to understanding the representation theory of $\mathcal{S}(m|n)_z$. Since this is built up from the finite-dimensional superalgebras, $\mathcal{S}(m|n;r,s)$, the first step is to understand their representation theory.
\section{Equivalence of Definitions and Semisimplicity}
In this section, $k$ will denote an algebraically closed field of characteristic zero, $V=k^{m|n}, W=V^*$, and  $T(r,s)=V^{\otimes r}\otimes W^{\otimes s}$. In the classical case, the symmetric group $S_r$ acts on the space $(k^n)^{\otimes r}$, and the Schur algebra $S(n,r)$ may be defined as $\operatorname{End}_{S_r}((k^n)^{\otimes r})$. This definition is equivalent to taking the dual of the coalgebra given by degree-$r$ homogeneous polynomials in $k[GL(n)]$. The Schur algebra $S(n,r)$ is also isomorphic to the image of the canonical $\gl(n)$ representation $\rho_r: U(\gl(n))\rightarrow \operatorname{End}((k^n)^{\otimes r})$. Thus, a Schur algebra may be defined through any of the three equivalent definitions. 
This section aims to demonstrate a similar equivalence of definitions for $\mathcal{S}(m|n;r,s)$. 
 First, we prove that $k[GL(m|n)]$ is isomorphic to the Hopf superalgebra of regular functions on $GL(m|n)$, which was introduced in \cite{SCHEUNERT200244}. This allows us to demonstrate that $\mathcal{S}(m|n;r,s)$ is isomorphic to the image of the universal enveloping superalgebra $U(\gl(m|n))$ under its natural action on $T(r,s)$. If $m-n\geq r+s$, then the semisimplicity of this image is shown by proving $T(r,s)$ is an irreducible $\mathcal{S}(m|n;r,s)$-supermodule on which the superalgebra acts faithfully. Lastly, it is known that the walled Brauer algebra $B_{r,s}(m-n)$ acts naturally on $T(r,s)$, and when $m-n\geq r+s$ it is isomorphic to $\operatorname{End}_{\gl(m|n)}(T(r,s))$ \cite{schurweyl}. The semisimplicity result in combination with this Schur-Weyl duality then implies through the double centralizer theorem that $\mathcal{S}(m|n;r,s)\cong \operatorname{End}_{B_{r,s}(m-n)}(T(r,s))$.

\subsection{The Hopf Superalgebra of Regular Functions on $GL(m|n)$}

 The Lie superalgebra $\frak{gl}(m|n)$ acts naturally on $T(r,s)$ through the following action defined on homogeneous elements:
\begin{equation*}
\begin{split}x\cdot (v_1\otimes ...\otimes v_r\otimes w_1\otimes...\otimes w_s)&=
\sum_{k=1}^{r}(-1)^{|x|\sum_{p=1}^k|v_p|}v_1\otimes ...\otimes (x\cdot v_k)\otimes...\otimes v_r\otimes ...\otimes w_s\\
&+\sum_{k=1}^{s}(-1)^{|x|(\sum_{p=1}^r|v_p|+\sum_{t=1}^{k}|w_t|)}v_1\otimes ...\otimes v_r\otimes...\otimes (x\cdot w_k)\otimes ...\otimes w_s 
\end{split}
\end{equation*}
We denote this representation by $\rho_{r,s}:U(\gl(m|n))\rightarrow \operatorname{End}_k(T(r,s))$.
\begin{definition}{\cite{SCHEUNERT200244}}

The \emph{Hopf superalgebra of regular functions on $GL(m|n)$} is the sub-superalgebra of $(U(\frak{gl}(m|n)))^*$ generated by: 
\begin{equation*}cf_{U(\frak{gl}(m|n))}(V)=\operatorname{span}_k\{t_{i,j}:1\leq i,j\leq m+n\}\quad \text{and} \quad cf_{U(\gl(m|n))}(W)=\operatorname{span}_k\{\overline{t}_{i,j}:1\leq i,j\leq m+n\}\end{equation*} 
\end{definition}

We denote this sub-superalgebra by $\mathcal{B}$. In \cite{SCHEUNERT200244} it was demonstrated that $\mathcal{B}$ is a Hopf superalgebra. We briefly recall its superbialgebra structure maps. Let $\Delta_0$ and $\epsilon_0$ be the canonical comultiplication   and counit structure maps on $U(\frak{gl}(m|n))$, respectively.
Let $f,g\in U(\frak{gl}(m|n))^{*}$ and $a,b\in U(\frak{gl}(m|n))$. The unit, multiplication, comultiplication, and counit of $\mathcal{B}$ are then all given respectively by:
\begin{equation*}1_{U(\frak{gl}(m|n))^*}:=\epsilon_0^*, \quad\langle m(f\otimes g),a\rangle := \langle f\otimes g, \Delta_0(a)\rangle,\quad \langle \Delta(f),a\otimes b\rangle :=\langle f, ab\rangle, \quad \epsilon(-):= \langle -,1_{U(\frak{gl}(m|n))}\rangle \end{equation*}

Furthermore,  let $B$ denote the superalgebra generated by $\{t_{i,j}:1\leq i,j\leq m+n \}$. It was shown in \cite[Theorem 2.1f]{muir1991polynomial} that $B$ is isomorphic to $\mathcal{A}(m|n)$ as a superbialgebra by sending $t_{i,j}\mapsto x_{i,j}$. Although, comultiplication on $\mathcal{A}(m|n)$ in this case is defined by:
\[\Delta '(x_{i,j}):=\sum_{h=1}^{m+n}(-1)^{(|i|+|h|)\cdot(|h|+|j|)}x_{i,h}\otimes x_{h,j}\]
 Denote $\mathcal{A}(m|n)$ with the superbialgebra structure given in our context by $(\mathcal{A}(m|n),\Delta)$, and the superbialgebra structure in \cite{muir1991polynomial} by $(\mathcal{A}(m|n),\Delta')$. It may be easily checked that the superalgebra isomorphism 
 \[f:(\mathcal{A}(m|n),\Delta)\rightarrow (\mathcal{A}(m|n),\Delta'),\quad x_{i,j}\mapsto (-1)^{|j|(|i|+|j|)}x_{i,j}\]
 is also a supercoalgebra isomorphism. Thus, $(\mathcal{A}(m|n),\Delta)$ and $B$ are isomorphic as superbialgebras. If $\alpha:=\det((t_{i,j})_{1\leq i,j\leq m})$ and $\beta:=\det((t_{i,j})_{m+1\leq i,j\leq m+n})$, then by localizing $B$ at $\alpha\cdot \beta$ it follows that $k[GL(m|n)]$ is isomorphic to $B_{\alpha\cdot \beta}$ as superbialgebras. Since a superbialgebra isomorphism of Hopf superalgebras is a Hopf superalgebra isomorphism, this implies that $k[GL(m|n)]$ is isomorphic to $B_{\alpha\cdot \beta}$. From \cite[Theorem 3.3]{SCHEUNERT200244} we see that $B_{\alpha\cdot \beta}$ is isomorphic to $\mathcal{B}$ as a Hopf superalgebra. So, we have the following lemma:
 \begin{lemma}
 The coordinate superalgebra $k[GL(m|n)]$ is isomorphic to $\mathcal{B}$ as a Hopf superalgebra.
 \end{lemma}

This lemma provides an alternative proof to Lemma $8.3$ from \cite{marko2020}. 

\begin{corollary}
\label{map_gen}
The coordinate superalgebra $k[GL(m|n)]$ is generated by $x_{i,j},\tilde{x}_{i,j}$ for $1\leq i,j\leq m+n$. So, $k[GL(m|n)]=\tilde{\mathcal{A}}(m|n)$.
\end{corollary}
\begin{proof}
Note, the Hopf superalgebra isomorphism given in \cite[Theorem 3.3]{SCHEUNERT200244} \[h:B_{\alpha\cdot \beta}\rightarrow \mathcal{B}\] sends $(-1)^{|j|(|i|+|j|)}x_{i,j}$ to $\tilde{t}_{i,j}:=(-1)^{|j|(|i|+|j|)}t_{i,j}$. Let $s'$ denote the antipode of $\mathcal{B}$. Then recalling \cite[pg.55]{SCHEUNERT200244}, we see
\[(s'((h\circ f)(x_{i,j})))_{1\leq i,j\leq m+n}= (\tilde{t}_{i,j})_{1\leq i,j\leq m+n}^{-1}\]
Since $h,f$ are Hopf superalgebra morphisms this implies:
\[(\tilde{t}_{i,j})_{1\leq i,j\leq m+n}^{-1}=((h\circ f)(s(x_{i,j})))_{1\leq i,j\leq m+n}=((h\circ f)(\tilde{x}_{i,j}))_{1\leq i,j\leq m+n}\]
But, 
\[(\tilde{t}_{i,j})_{1\leq i,j\leq m+n}^{-1}=(\overline{t}_{i,j})_{1\leq i,j\leq m+n}\]
Therefore, 
\[\overline{t}_{i,j}=(h\circ f)(\tilde{x}_{i,j})\quad\text{for}\quad 1\leq i,j\leq m+n\]
Since $\{t_{i,j},\overline{t}_{i,j}:1\leq i,j\leq m+n\}$ generates $\mathcal{B}$, this implies $\{x_{i,j},\tilde{x}_{i,j}:1\leq i,j\leq m+n\}$ generates $k[GL(m|n)]$. Recalling, Lemma \ref{description1} we see $k[GL(m|n)]=\tilde{\mathcal{A}}(m|n)$.
\end{proof}
\subsection{First equivalence of Definitions}

To show that a rational Schur superalgebra $\mathcal{S}(m|n;r,s)$ is isomorphic to $\rho_{r,s}(U(\gl(m|n)))$ we need the following result:
\begin{proposition}
\label{cf_iso}
As supercoalgebras $\tilde{\mathcal{A}}(m|n;r,s)$ is isomorphic to $\bold{cf}_{U(\frak{gl}(m|n))}(T(r,s))$. 
\end{proposition}
\begin{proof}
By the isomorphism given in Corollary \ref{map_gen} and Lemma \ref{description3}, it suffices to show that 
\begin{equation*}\bold{cf}_{U(\frak{gl}(m|n))}(T(r,s))=\operatorname{span}\{t_{i_1,j_1}^{m_1}...t_{i_p,j_p}^{m_p}\overline{t}_{i_{p+1},j_{p+1}}^{m_{p+1}}...\overline{t}_{i_{p+\ell},j_{p+\ell}}^{m_{p+\ell}}: \sum_{k=1}^{p}m_k=r,\sum_{k=1}^{\ell}m_{p+k}=s\}\end{equation*}

Recalling the definition of $\mathcal{B}$ this follows immediately from the following fact: If $A,B$ are finite-dimensional representations of $\frak{gl}(m|n)$, then 
\[ \bold{cf}_{U(\frak{gl}(m|n))}(A\otimes B)= \bold{cf}_{U(\frak{gl}(m|n))}(A)\cdot  \bold{cf}_{U(\frak{gl}(m|n))}(B)\] 

\end{proof}

\begin{theorem}{(Compare with \cite[Theorem 3.2]{SCHEUNERT200244})}
The rational Schur superalgebra $\mathcal{S}(m|n;r,s)$ is isomorphic to $\rho_{r,s}(U(\frak{gl}(m|n)))$.

\end{theorem}
\begin{proof}
This follows immediately from combining Proposition \ref{cf_iso} and Lemma \ref{coeff_alg_iso}.
\end{proof}

\subsection{Rational Schur Superalgebras as Centralizer Algebras and a Semisimplicity Result}

Recall that if $m\geq n$, then the walled Brauer algebra $B_{r,s}(m-n)$ acts on $T(r,s)$, and when $m-n\geq r+s$ it is isomorphic to $\operatorname{End}_{U(\frak{gl}(m|n))}(T(r,s))$ \cite{schurweyl}. 
Proceeding similarly to the classical case we will demonstrate that a rational Schur superalgebra is isomorphic to the centralizer of the action of a walled Brauer algebra, or in symbols
\[\mathcal{S}(m|n;r,s)\cong \operatorname{End}_{B_{r,s}(m-n)}(T(r,s))\]
By the previous section, it suffices to show that $\mathcal{S}(m|n;r,s)$ is semisimple, as once this has been proven we may apply the double centralizer theorem \cite[Theorem 11.1.1]{MR2906817} to obtain the desired isomorphism. To do this it suffices to find a semisimple faithful representation of $\mathcal{S}(m|n;r,s)$. A natural candidate for this is $T(r,s)$.

It is known that the $T(r,s)$ decomposes as a $\gl(m|n)$-supermodule into indecomposable sub-supermodules and hence indecomposable $\mathcal{S}(m|n;r,s)$ sub-supermodules. Brundan and Stroppel classified the isomorphism types of these indecomposables in \cite{BRUNDAN} through $(m|n)$-cross bi-partitions. 
\begin{definition}
If $\lambda_L$ is a partition of $r$ and $\lambda_R$ a partition of $s$, then $(\lambda_L,\lambda_R)$ is a \emph{$(m|n)$-cross bi-partition of $(r,s)$} if there exists a $1\leq i\leq m+1$ such that $\lambda^L_i+\lambda^R_{m+2-i}<n+1$. 
\end{definition}

The set of $(m|n)$-cross bi-partitions of $(r,s)$ is denoted by $\Lambda_{r,s}$.
\begin{theorem}{\cite[Theorem 8.19]{BRUNDAN}}
The indecomposable $\mathcal{S}(m|n;r,s)$ sub-supermodules of $T(r,s)$ are parameterized up to isomorphism by the $(m|n)$-cross bi-partitions of $(r,s)$. 
\label{classify}
\end{theorem}

Suppose $\lambda\in \Lambda_{r,s}$, and let $R(\lambda)$ denote the unique (up to isomorphism) $\frak{gl}(m|n)$ sub-supermodule of $T(r,s)$ associated to $\lambda$. In \cite{Heidersdorf}, Heidersdorf associates to each such bi-partition an invariant $d(\lambda)$ and shows that $R(\lambda)$ is an irreducible $\frak{gl}(m|n)$-supermodule if and only if $d(\lambda)=0$. Furthermore, Proposition 10.3 of the same paper states that if $ m-n\geq r+s$, then $d(\lambda)=0$. Since we are assuming $ m-n\geq r+s$ so that $B_{r,s}(m-n)$ is isomorphic to $\operatorname{End}_{U(\frak{gl}(m|n))}(T(r,s))$, we may conclude that all indecomposable $\frak{gl}(m|n)$ sub-supermodules of $T(r,s)$ are also irreducible. Summarizing this discussion we have:
\begin{lemma}
If $m-n\geq r+s$, then $T(r,s)$ is a semisimple $\frak{gl}(m|n)$-supermodule. 
\end{lemma}

Since $\mathcal{S}(m|n;r,s)$ acts as $\rho_{r,s}(U(\gl(m|n)))$ on $T(r,s)$, this is evidently a faithful semisimple representation. Therefore, $\mathcal{S}(m|n;r,s)$ is semisimple as an algebra.  
\begin{remark}
Assuming \cite[Proposition 1.7 ]{semisimplesuper}, this implies that $\mathcal{S}(m|n;r,s)$ is a semisimple superalgebra.
\end{remark}

\begin{theorem}

If $m-n\geq r+s $, then $\mathcal{S}(m|n;r,s)$ is the centralizer algebra of the action of $B_{r,s}(m-n)$, or in symbols 
\[\mathcal{S}(m|n;r,s)\cong \operatorname{End}_{B_{r,s}(m-n)}(T(r,s))\]
\end{theorem}
\begin{proof}

By mixed Schur-Weyl duality it is known that $\operatorname{End}_{\gl(m|n)}(T(r,s))\cong \rho(B_{r,s}(m-n))$ where $\rho$ denotes the right action of $B_{r,s}(m-n)$ on $T(r,s)$. 
By the double centralizer theorem \cite[Theorem 11.1.1]{MR2906817} this implies that $\operatorname{End}_{B_{r,s}(m-n)}(T(r,s))=\rho_{r,s}(U(\gl(m|n)))$. 

\end{proof}

\section*{Acknowledgements }

I would like to thank Dr. Nicolas Guay for his invaluable guidance and support in this project.


\bibliography{RSS_ref}
\bibliographystyle{algebra_coliq}
\end{document}